\documentclass[reqno]{amsart}
\usepackage{amssymb,amsmath,amsthm,mathrsfs, array}
\usepackage{amsfonts,latexsym,amsxtra,amscd}
\usepackage{pgfplots, tikz}
\usepackage{hyperref}
\usepackage{graphicx, stmaryrd}
\usepackage[centering]{geometry}

\newtheorem{thm}{Theorem}[section]

\newtheorem{cor}[thm]{Corollary}
\newtheorem{lem}[thm]{Lemma}

\newtheorem{rem}[thm]{Remark}

\numberwithin{equation}{section}

\begin{document}

\title[Fourier decay bound and differential images of self-similar measures]
{Fourier decay bound and differential images of self-similar measures}


\author{Yuanyang Chang}
\address{Department of Mathematics, South China University of Technology, Guangzhou, 510640, P. R. China}
\email{chrischang2016@gmail.com}

\author{Xiang Gao$^{\ast}$}
\address{School of Mathematics and Statistics, Wuhan University, Wuhan, 430072, P. R. China}
\email{gaojiaou@gmail.com}
 \thanks{$^{\ast}$corresponding author.}

\date{}
\maketitle{}

\begin{abstract}
In this note, we investigate $C^2$ differential images of the homogeneous self-similar measure associated with an IFS $\mathcal{I}=\{\rho x+a_j\}_{j=1}^m$ satisfying the strong separation condition and a positive probability vector $\vec{p}$. It is shown that the Fourier transforms of such image measures have power decay for any contractive ratio $\rho\in (0, 1/m)$, any translation vector $\vec{a}=(a_1, \ldots, a_m)$ and probability vector $\vec{p}$, which extends a result of Kaufman on Bernoulli convolutions. Our proof relies on a key combinatorial lemma originated from Erd\H{o}s, which is important in estimating the oscillatory integrals. An application to the existence of normal numbers in fractals is also given.
\end{abstract}

\bigskip
\section{Introduction}
Let $\mu$ be a Borel probability measure on $\mathbb{R}$. The Fourier transform of $\mu$ is defined by
\begin{equation}\label{Fourier}
 \widehat{\mu}(\xi)=\int \exp(2\pi i\xi t)d\mu(t), \;\; \xi\in\mathbb{R}.
\end{equation}
It is known that the asymptotic behaviors of $\widehat{\mu}(\xi)$ at infinity give information to absolute continuity or singularity, the geometric or arithmetic structure, and the size of the support of $\mu$ (see for instance \cite{Mat15} for more details). In this note, we are concerned with a special class of Borel probability measures, that is, the differential images of some homogeneous self-similar measures. The motivation is to see how the asymptotic behaviors of the Fourier transform are affected when a self-similar measure is smoothly perturbated.

\vspace{1mm}
Recall that an iterated function system (IFS) is a finite family $\mathcal{I}=\left\{f_j\right\}_{j=1}^m (m \geq 2)$ of strictly contracting maps on some complete metric space $X$ (here we always assume that $X=\mathbb{R}$). An IFS $\mathcal{I}$ is called linear (resp. of class $C^{\alpha}$) if all the maps in $\mathcal{I}$ are affine (resp. of class $C^{\alpha}$). An IFS $\mathcal{I}$ is called $C^{\alpha}$-conjugate to another IFS $\mathcal{J}$ if $\mathcal{J}=\left\{\varphi\circ f_j \circ\varphi^{-1}\right\}_{j=1}^m$ for some $C^{\alpha}$-diffeomorphism $\varphi$. It is well known that there exists a unique non-empty compact set $K=K(\mathcal{I})\subset \mathbb{R} $, called the attractor of $\mathcal{I}$, such that
$K=\bigcup_{j=1}^{m}f_j(K)$. We say that the strong separation condition holds for $\mathcal{I}$ if the pieces $\left\{f_j(K)\right\}_{j=1}^m$ are pariwise disjoint. For a measurable map $f:\mathbb{R}\rightarrow\mathbb{R}$ and a Borel probability measure $\mu$, we use $f \mu$ to denote the push-forward measure of $\mu$ by $f$, that is, $f \mu(A)=\mu\left(f^{-1}(A)\right)$ for all Borel subsets $A\subset\mathbb{R}$. $f \mu$ is also called the image measure of $\mu$ under $f$. Given an IFS $\mathcal{I}=\left\{f_j\right\}_{j=1}^m$ and a positive probability vector $\vec{p}=(p_1, \ldots, p_m)$, namely, $\sum_{j=1}^m p_j=1$ and $p_j>0$ for each $j$, there exists a unique Borel probability measure $\mu$ supported on $K$ such that
\begin{equation}\label{selfsim}
 \mu=\sum_{j=1}^{m} p_j \cdot {f_j}\mu.
\end{equation}

In this note we mainly consider linear IFSs, that is, $f_j(x)=\rho_j x+a_j$ with $0<\rho_j<1$ and $a_j\in\mathbb{R}$ for all $j=1,\ldots, m$. Then each $f_j$ is a similitude in the sense that $|f_j(x)-f_j(y)|=\rho_j |x-y|$ for any $x, y\in\mathbb{R}$. The attractor $K$ and the measure $\mu$ satisfying \eqref{selfsim} are termed as a self-similar set and self-similar measure respectively. We are particularly interested in IFS of the form $\mathcal{I}=\{\rho x+a_j\}_{j=1}^m$, where $\rho\in (0, 1)$ is the common contractive ratio and $\vec{a}=(a_1,\ldots, a_m)$ denotes the translation vector. We call the corresponding self-similar measure $\mu_{\rho}^{\vec{p}, \vec{a}}$ a homogenous self-similar measure and write it as $\mu_{\rho}$ for simplicity. When $m=2$ and $p=(1/2, 1/2)$, $\mu_{\rho}$ is the so-called (infinite) Bernoulli convolution, which was studied extensively but is still not completely understood (see \cite{PSS00} for a good survey).  It is well known that $\mu_{\rho}$ can be viewed as the weak limit of the $(N+1)$-fold convolution product $\ast_{k=0}^N (p_1\delta_{\rho^k a_1}+\cdots +p_m\delta_{\rho^k a_m})$ \cite{Mat15}. Moreover, the Fourier transform $\widehat{\mu_{\rho}}$ has the form of infinite product:
\begin{equation}\label{Fourier-sim}
  \widehat{\mu_{\rho}}(\xi)=\int \exp(2\pi i  \xi t) d\mu_{\rho}(t)=\prod_{k=0}^{\infty}\sum_{j=1}^{m}p_j\exp(2\pi i  a_j \xi \rho^{k}), \; \xi\in\mathbb{R}.
\end{equation}

Asymptotic behaviors of $\widehat{\mu_{\rho}}$ were studied by many authors. For the Bernoulli convolutions, the Erd\H{o}s-Salem theorem
\cite{Er39, Sal63} states that $\widehat{\mu_{\rho}}$ has no decay at infinity if and only if $\rho^{-1}$ is a Pisot number and $\rho\neq 1/2$ (recall that an algebraic integer $\beta>1$ is called a Pisot number if all its conjugate roots have modulus strictly less than one). For general homogeneous self-similar measures with $m>2$, Hu \cite{Hu01} obtained similar criterions in some special cases, where the decay properties of $\widehat{\mu_{\rho}}$ is also affected by the translation vector $\vec{a}$ and the probability vector $\vec{p}$. Explicit examples of $\mu_{\rho}$ whose Fourier transforms have certain decay rate were given in \cite{ BS14, Gar62, Ker36} (Bernoulli convolutions) and \cite{DFW07, GM17, Wat12} (general homogeneous self-similar measures). Note that in all cases the arithmetic properties of $\rho$ play an important role. Furthermore, Shmerkin \cite{Shm14} showed that there exists a set $E\subset (0, 1)$ of Hausdorff dimension 0 such that for all $\rho\in (0, 1)\setminus E$, all translation vectors $\vec{a}$ with distinct coordinates and all positive probability vector $\vec{p}$, the Fourier transform $|\widehat{\mu_{\rho}}(\xi)|$ tends to 0 polynomially.

\vspace{1mm}
In this note, we study the Fourier decay properties of the image measure of $\mu_{\rho}$ under a differential function. The strong separation condition is assumed to hold for the corresponding IFS, thus $\rho\in (0, 1/m)$ and $\mu_{\rho}$ is supported on a set of Cantor type.

\begin{thm}\label{mainthm}
Let $\mu_{\rho}$ be the homogeneous self-similar measure associated with an IFS $\mathcal{I}=\{\rho x+a_j\}_{j=1}^m$ that satisfies the strong separation condition and a positive probability vector $\vec{p}$. Then for any $\varphi\in C^2(\mathbb{R})$ with $\varphi''>0$ and any $g\in C^1(\mathbb{R})$, we have
\begin{equation}\label{osc-int}
  \Big|\int \exp(2\pi i\xi \varphi(t)) g(t) d\mu_{\rho}(t)\Big|=O(|\xi|^{-\gamma}), \; |\xi|\rightarrow \infty
\end{equation}
for some constant $\gamma>0$. In particular, the Fourier transform $\widehat{\varphi \mu_{\rho}}$ of the image measure $\varphi \mu_{\rho}$ has a power decay.
\end{thm}

\begin{rem}
 \begin{enumerate}
\item The above theorem indicates that the Fourier decay bound of $\widehat{\varphi \mu_{\rho}}$ is polynomial for any parameters $\rho, \vec{p}, \vec{a}$,
      even though the Fourier transform $\widehat{\mu_{\rho}}$ of the original measure $\mu_{\rho}$ admits no decay at infinity. This means that smooth permutations could change the Fourier decay behaviors of measures, where the non-linearity of the phase function is crucial for obtaining the power decay (otherwise one can easily find a counterexample, for instance, $\varphi(x)=x$ and $\mu_{\rho}$ is the Cantor-Lebesgue measure on the middle-third Cantor set). We remark that the theorem also holds in the case of $\varphi''<0$, where the proof needs only a slight modification. It seems that the condition could be weakened as long as the graph of $\varphi$ is curved enough.

\item The fact that the Fourier decay properties of push-forward measures are closely related to the smoothness of the phase functions date back to van der
      Corput's lemma in harmonic analysis (see \cite{Mat15}). Kaufman \cite{Kau84}, by rediscovering a method of Erd\H{o}s \cite{Er40}, proved that the Fourier transform of $C^2$ differential images of a class of Bernoulli convolutions has power decay. Applying the same method, we extend Kaufman's result to general homogeneous self-similar measures. Moreover, our proof shows that the constant $\gamma$ in \eqref{osc-int} could be taken explicitly. In fact, one can choose appropriate $\beta\in (1/2, 1)$ and $\epsilon\in (0, \delta)$ with $(2-\alpha)\beta<1$ such that
      $\min\{2\beta-1, \frac{(1-\beta)\epsilon\log\delta}{\log \rho}\}$ attains the maximum, and then let $\gamma$ be this maximum.
 \end{enumerate}
\end{rem}

The main ingredients of the proof  are the convolution nature of the homogeneous self-similar measure $\mu_{\rho}$ and a combinatorial lemma
originated from  Erd\H{o}s \cite{Er40}. The former allows us to approximate $\mu_{\rho}$ by a sequence of discrete measures; the latter can be used to control the ``bad'' set for the purpose of estimating the oscillatory integral. Recently, by using an estimate on decay of exponential sums due to Bourgain, Bourgain-Dyatlov \cite{BD17} proved similar results for the Patterson-Sullivan measure on the limit set $\Lambda_{\Gamma}$ of a co-compact subgroup $\Gamma\leq SL(2, \mathbb{Z})$, where the nonlinearity of transformations in $\Gamma$ is very important. Measures whose Fourier transforms admit power decay were also studied for stationary measures \cite{Li17} and for Gibbs measure with respect to the Gauss map \cite{JS16, Kau81, QR03}.

\smallskip
This note is part of the second author's PhD thesis (in Chinese). When we are going to submit this note, we hear that a version of Theorems \ref{mainthm} is simultaneously and independently proved by Mosquera and Shmerkin (see Theorem 3.1 of \cite{MS17}) among some other results. However, we remark that our method
allows us to obtain a better decay bound (see Remark 1.1.(2)), besides the combinatorial lemma established in this note is different from theirs.

\smallskip
The asymptotic behavior of $\widehat{\mu_{\rho}}$ is closely related to the arithmetic structure of $\mu_{\rho}$. We give an application of Theorem \ref{mainthm} in this spirit. Recall that a sequence $\{x_n\}_{n\geq 1}$ in $\mathbb{R}$ is called equidistributed modulo 1 if for any subinterval $I\subset [0, 1]$, we have
\begin{equation*}
 \lim\limits_{N\rightarrow\infty}\frac{1}{N}\#\big\{n: 1\leq n \leq N, \{x_n\}\in I\big\}=|I|,
\end{equation*}
where $\{x_n\}$ stands for the fractional part of $x_n$ and $|I|$ denotes the length of $I$. For a fixed integer $b>1$, if the sequence $\{b^n x\}_{n\geq 1}$ is equidistributed modulo 1, then we say that $x$ is normal in base $b$, or equivalently $b$-normal. Weyl's equidistribution criterion provides a useful tool to deal with the problem of normality. Motivated by this, Davenport-Erd\H{o}s-Leveque \cite{DEL63} established a criterion to check whether a random number with respect to a Borel probability measure $\mu$ is normal or not. It was reformulated by Queff\'{e}lec-Ramar\'{e} \cite{QR03} as follows:
\begin{thm}\label{DEL}(Davenport-Erd\H{o}s-Leveque \cite{DEL63, QR03})
Let $\mu$ be a Borel probability measure on $\mathbb{R}$ and $\{s_n\}_{n\geq 1}$ a sequence of positive integers. If for any integer $h\neq 0$,
\begin{equation}\label{DEL-formula}
  \sum_{N=1}^{\infty}\frac{1}{N^3}\sum_{m, n=1}^{N}\widehat{\mu}(h(s_n-s_m))<\infty,
\end{equation}
then the sequence $\{s_n x\}_{n\geq 1}$ is equidistributed modulo 1 for $\mu$ almost every $x$. \par
In particular, when $|\widehat{\mu}(\xi)|=O(|\xi|^{-\gamma})$ for some $\gamma>0$ as $|\xi|\rightarrow\infty$, then \eqref{DEL-formula} holds for every strictly increasing sequence $\{s_n\}_{n\geq 1}$, which implies that $\mu$ almost every $x$ is normal in any bases.
\end{thm}

Indeed, in \cite{QR03}, Theorem \ref{DEL} was stated for the Borel probability measures on $[0, 1]$. However, the same result holds for any Borel probability measures on $\mathbb{R}$ with a slight modification of their proof.  The next corollary is an immediate consequence of the combination of Theorem \ref{mainthm} and \ref{DEL}, which reproduces some results of Hochman-Shmerkin on the existence of normal numbers in fractals (see Section 1.3.1 in \cite{HS15}).

\begin{cor}
  Let $\mu_{\rho}$ and $\varphi$ be as in Theorem 1.1, then for $\mu_{\rho}$ almost every $x$, $\{s_n \varphi(x)\}_{n\geq 1}$ is equidistributed modulo 1
for any strictly increasing sequence $\{s_n\}_{n\geq 1}$. In particular, for $\mu_{\rho}$ almost every $x$, $\varphi(x)$ is $b$-normal for every base $b$.
\end{cor}

In fact, Hochman-Shmerkin presented a general criterion for the $b$-normality of a measure $\mu$, which adapts to many general nonlinear iterated function
systems (see Theorem 1.4-1.7 in \cite{HS15}). However, if $\varphi$ is a homomorphism, then $\varphi \mu_{\rho}$ is a self-conformal measure associated with the IFS $\{\varphi^{-1}\circ f_j \circ \varphi\}_{j=1}^m$, which is $C^2$-conjugate to the linear IFS $\{f_j\}_{j=1}^m$. Thus, our result extends Theorem 1.6 in \cite{HS15} in this conjugate-to-linear case by weakening the $C^{\omega}$ condition.

Hochman-Shmerkin's approach is ergodic, i.e., to study the ergodic properties of the so-called scenery flow invariant distributions generated by the Gibbs measure; while our method is analytic, which allows us to obtain more quantitative information on the the decay of the Fourier transform of a certain measure. Moreover, the result of Hochman-Shmerkin does not apply to give more general statements on the equidistribution modulo 1 of $\{s_n \varphi(x)\}_{n\geq 1}$ for any strictly increasing sequence $\{s_n\}_{n\geq 1}$, so our result extends theirs in this respect.


\section{A Combinatorial lemma}
In this section, we give a combinatorial lemma originated from Erd\H{o}s \cite{Er40} (see also \cite{Kah69, Kau84}), which is important
in the estimation of the oscillatory integral in \eqref{osc-int}. Here and below, $\|x\|$ denotes the distance from a real number $x$ to its nearest integer.
Let $c_0$ be a fixed real number and $\theta>1$. For any $\epsilon>0$ and fixed closed interval $[H_1, H_2]\subset\mathbb{R}$, let
\begin{equation*}
 \Gamma(\epsilon)=\Big\{x\in [H_1,H_2]:\#\big\{1\leq k\leq N:\; \|c_0 \theta^k x\|< \frac{1}{2(1+\theta)}\big\}> (1-\epsilon) N \Big\}.
\end{equation*}
The combinatorial lemma is stated as follows:
\begin{lem}\label{com-lem}
 For any $\epsilon\in (0, 1/2)$, $\Gamma(\epsilon)$ could be covered by $O(\exp(\omega(\epsilon)N))$ intervals of length
$c_0^{-1}\theta^{-N}$ for $N$ large enough, where $\omega(\epsilon)=-\epsilon\log\epsilon-(1-\epsilon)\log(1-\epsilon)+2\epsilon\log(\theta+2)$.
\end{lem}

The lemma says that the set of parameters $x$ such that $c_0 \theta^k x$ is close to an integer for ``most'' indices $k\in\{1,\ldots, N\}$ is very small.
It turns out that the ``bad'' set for the sake of estimating the oscillatory integral in \eqref{osc-int} can be controlled by sets of this kind.
For each $x\in [H_1, H_2]$, write
\begin{equation*}
 c_0 \theta^k x=r_k(x)+\varepsilon_k(x):=r_k+\varepsilon_k,
\end{equation*}
where $r_{k}\in\mathbb{Z}$ and $-\frac{1}{2}<\varepsilon_{k}\leq\frac{1}{2}$. The proof of Lemma \ref{com-lem} relies on the next lemma.

\begin{lem}\label{lem}
 For each value of $r_k$, there are at most $\theta $ + 2 possible choices of $r_{k+1}$. Moreover, if
$\max\left\{|\varepsilon_{k}|,|\varepsilon_{k+1}|\right\}<\frac{1}{2(1+\theta)}$, then $r_{k+1}$ is uniquely determined by $r_k$.
\end{lem}

\begin{proof}
Since $c_0 \theta^k x= r_k+\varepsilon_k, \; c_0 \theta^{k+1} x= r_{k+1}+\varepsilon_{k+1}$, we have
\begin{equation*}
 |r_{k+1}-\theta r_{k}|=|\varepsilon_{k+1}-\theta\varepsilon_{k}|.
\end{equation*}
Thus for a given value of $r_k$, $r_{k+1}$ falls into the interval $[\theta r_{k}-\frac{\theta+1}{2},\theta r_{k}-\frac{\theta+1}{2}]$, which means that $r_{k+1}$ has at most $\theta $+2 choices.

\vspace{1mm}
Further, if $\max\left\{|\varepsilon_{k}|,|\varepsilon_{k+1}|\right\}<\frac{1}{2(1+\theta)}$, then
\begin{equation*}
 |r_{k+1}-\theta r_{k}|=|\varepsilon_{k+1}-\theta\varepsilon_{k}|\leq(1+\theta)\max\left\{|\varepsilon_{k}|,|\varepsilon_{k+1}|\right\}<1/2.
\end{equation*}
It follows that $r_{k+1}$ is uniquely determined by $r_k$.
\end{proof}

\noindent\textbf{Proof of Lemma 2.1}
Observe that $x=\frac{r_N(x)}{c_0 \theta^N}+\frac{\epsilon_N(x)}{c_0 \theta^N}\in \big(\frac{r_N(x)}{c_0 \theta^N}-\frac{1}{2 c_0 \theta^N},\frac{r_N(x)}{c_0 \theta^N}+\frac{1}{2 c_0 \theta^N}\big]$. Fix $\epsilon$, we first estimate the numbers of integers $r_N$ such that $r_N=r_N(x)$ for some
$x\in\Gamma(\epsilon)$.

\vspace{0.5mm}
To begin with, for a fixed $x\in\Gamma(\epsilon)$, we have $\#\big\{1\leq k\leq N:\; \|c_0 \theta^k x\|\geq \frac{1}{2(1+\theta)}\big\}< \epsilon N$.
Let $1 \leq i_1<\ldots<i_s\leq N$ be all of the indices $k$ such that $\|c_0 \theta^k x\|\geq \frac{1}{2(1+\theta)}$, or equivalently,  $|\varepsilon_{k}(x)|\geq \frac{1}{2(1+\theta)}$. We know from the definition that $s<\epsilon N$, thus the number of choices of such
$\{i_1, \ldots, i_s\}$ can not exceed
\begin{equation*}
  \sum\limits_{j=0}^{[\epsilon N]+1}\binom{N}{j} \ll \exp\big(h(\epsilon) N+o(N)\big),
\end{equation*}
where the last inequality is by Stirling's formula and $h(t)=-t\log t-(1-t)\log(1-t)$.

\vspace{0.5mm}
On the other hand, take a $\{i_1, \ldots, i_s\}$ as above and fix it, by Lemma \ref{lem}, if $k$ is of the form $i_l$ or $i_l+1$ for some $l=1,\ldots,s$, then there are at most $\theta +2$ possible choices of $r_k$ after $r_{k-1}$ has been given; conversely, if $k$ is not of the form $i_l$ or $i_l+1$ for any $l=1,\ldots,s$, then $r_k$ is uniquely determined by $r_{k-1}$. Note that the number of choices of $r_1$ is less than $c_0\theta(H_2-H_1)+1$, thus for each specific $\{i_1, \ldots, i_s\}$, there exists at most $(\theta +2)^{2s}\big(c_0\theta(H_2-H_1)+1\big)$ values of $r_N$ such that $r_N=r_N(x)$ for some $x$ with $\#\big\{1\leq k\leq N:\; \|c_0 \theta^k x\|\geq \frac{1}{2(1+\theta)}\big\}=s<\epsilon N$.

\vspace{0.5mm}
Therefore, the number of distinct integers $r_N$ such that $r_N=r_N(x)$ for some $x\in\Gamma(\epsilon)$ is less than
\begin{equation*}
 \exp\big(N h(\epsilon)+o(N)\big)(\theta +2)^{2\epsilon N}\big(c_0\theta(H_2-H_1)+1\big) \ll \exp(\omega(\epsilon) N+o(N)).
\end{equation*}
Since each such $r_N$ corresponds to a unique interval of length $c_0^{-1}\theta^{-N}$, we conclude the proof. \par
\qed

\section{Proof of Theorem 1.1}
The proof is divided into three steps:

\vspace{0.5mm}
\textbf{Step 1.} Decomposition of the self-similar measure $\mu_{\rho}$. \par
\vspace{0.5mm}
In this step, we decompose $\mu_{\rho}$ by virtue of its convolution nature, which allows us to do subtle analysis. Assume
$p_l=\max_{1\leq j\leq m} p_j$ and $p_s=\min_{1\leq j\leq m} p_j$, let $\alpha=\frac{\log p_l}{\log\rho}$. Without loss of generality, assume $a_l>a_s$.
Notice that $p_l\in [1/m, 1)$ and $\rho\in (0, 1/m)$, we have $0<\alpha<1$. Take
$\beta\in (1/2, 1)$ such that $(2-\alpha)\beta<1$. Let $|\xi|$ be fixed and large enough. Let $N_1,N_2$ be two positive integers satisfying
\begin{equation*}
 \rho^{-N_1}\leq |\xi|^{\beta}<\rho^{-N_1-1},\quad\rho^{-N_2}\leq |\xi|<\rho^{-N_2-1}.
\end{equation*}
We decompose $\mu_{\rho}$ into $\mu_{\rho}=\mu_{N_1}\ast\eta_{N_1}$, where $\ast$ stands for the convolution operator between measures,
$\mu_{N_1}=\ast_{k=0}^{N_1}(p_1\delta_{\rho^k a_1}+\cdots +p_m\delta_{\rho^k a_m})$ and
$\eta_{N_1}=\ast_{k=N_1+1}^{\infty}(p_1\delta_{\rho^k a_1}+\cdots +p_m\delta_{\rho^k a_m})$.
Then the oscillatory integral in \eqref{osc-int} could be written as
\begin{equation}\label{decomposition}
 \begin{split}
       &\int \exp(2\pi i\xi \varphi(t)) g(t) d\mu_{\rho}(t)\\
      =&\int \exp(2\pi i \xi \varphi(t)) g(t) d\mu_{N_1}\ast\eta_{N_1}(t)\\
      =&\iint \exp(2\pi i \xi \varphi(x+y)) g(x+y) d\mu_{N_1}(x)d\eta_{N_1}(y).
 \end{split}
\end{equation}

\vspace{0.5mm}
\textbf{Step 2.} Linearizing the phase $\varphi$. \par
\vspace{0.5mm}
Let $\widetilde{K}\subset\mathbb{R}$ be the minimal closed interval containing the support of $\mu$.
Put $H_0=\sup_{t\in \widetilde{K}}|\varphi''(t)|$ and $M=\sup_{t\in 2\widetilde{K}}\big\{|g(t)|+|g'(t)|\big\}$.
Split the double integral in \eqref{decomposition} into three parts:
\begin{equation}\label{3parts}
   \iint \exp(2\pi i \xi \varphi(x+y)) g(x+y) d\mu_{N_1}(x)d\eta_{N_1}(y)=\Upsilon_1+\Upsilon_2+\Upsilon_3,
\end{equation}
where
\begin{equation*}
 \Upsilon_1=\iint \Big(\exp\big(2\pi i \xi \varphi(x+y)\big)-\exp\big(2\pi i \xi (\varphi(x)+\varphi'(x)y)\big)\Big)g(x+y) d\mu_{N_1}(x)d\eta_{N_1}(y),
\end{equation*}

\begin{equation*}
 \Upsilon_2=\iint \exp\big(2\pi i \xi (\varphi(x)+\varphi'(x)y)\big)\big(g(x+y)-g(x)\big) d\mu_{N_1}(x)d\eta_{N_1}(y),
\end{equation*}
and
\begin{equation*}
 \Upsilon_3=\iint \exp\big(2\pi i \xi (\varphi(x)+\varphi'(x)y)\big) g(x) d\mu_{N_1}(x)d\eta_{N_1}(y).
\end{equation*}
By Taylor's formula,
\begin{equation*}
 |\varphi(x+y)-\varphi(x)-\varphi'(x)y|< H_0 y^2.
\end{equation*}
Combining the inequality $|\exp(2\pi i x)-1|\leq 2\pi |x|$, we have
\begin{equation*}
 \begin{split}
    |\Upsilon_1|\leq & \iint M \big|\exp\big(2\pi i \xi (\varphi(x+y)-\varphi(x)-\varphi'(x)y)\big)-1\big| d\mu_{N_1}(x)d\eta_{N_1}(y) \\
         \leq & \iint 2\pi M |\xi|\cdot\big|\varphi(x+y)-\varphi(x)-\varphi'(x)y\big| d\mu_{N_1}(x)d\eta_{N_1}(y)\\
         \leq & \iint 2\pi M H_0 |\xi| y^2 d\mu_{N_1}(x)d\eta_{N_1}(y).
 \end{split}
\end{equation*}
Since $\eta_{N_1}$ is supported on the set of points $\big\{\sum_{k=N_1+1}^{\infty}\rho^k X_k\big\}$ with $X_k\in\{a_1,\ldots, a_m\}$ for each $k$. Thus the support of $\eta_{N_1}$ is of length $O(\rho^{N_1+1})=O(|\xi|^{-\beta})$, it follows that
\begin{equation}\label{I_1}
  |\Upsilon_1|=O(|\xi|\cdot|\xi|^{-2\beta})=O(|\xi|^{1-2\beta}).
\end{equation}

For $\Upsilon_2$, similarly as $\Upsilon_1$, we have
\begin{equation}\label{I_2}
  \begin{split}
    |\Upsilon_2|\leq & \iint \big|g(x+y)-g(x)\big| d\mu_{N_1}(x)d\eta_{N_1}(y) \\
         \leq & \iint M y d\mu_{N_1}(x)d\eta_{N_1}(y)=O(|\xi|^{-\beta}).
 \end{split}
\end{equation}

For $\Upsilon_3$, by Fubini's theorem,
\begin{equation*}
  \begin{split}
|\Upsilon_3|
           = & \Big|\int \exp\big(2\pi i \xi \varphi(x)\big) g(x) \Big(\int \exp\big(2\pi i \xi \varphi'(x)y\big)d\eta_{N_1}(y)\Big) d\mu_{N_1}(x)\Big| \\
           = & \Big|\int \exp\big(2\pi i \xi \varphi(x)\big) g(x) \widehat{\eta_{N_1}}(\xi\varphi'(x)) d\mu_{N_1}(x)\Big| \\
        \leq & \int M \big|\widehat{\eta_{N_1}}(\xi\varphi'(x))\big| d\mu_{N_1}(x).
 \end{split}
\end{equation*}
The convolution nature of $\eta_{N_1}$ implies
\begin{equation*}
 |\widehat{\eta_{N_1}}(\xi \varphi'(x))|=\prod_{k=N_1+1}^{\infty}|\Phi(\xi \varphi'(x) \rho^k)|,
\end{equation*}
where $\Phi(t)=\sum_{j=1}^m p_j\exp(2\pi i  a_j t)$. Therefore,
\begin{equation*}
 |\Upsilon_3|\leq M \int \big(\prod_{k=N_1+1}^{\infty}|\Phi(\xi\varphi'(x) \rho^k)|\big) d\mu_{N_1}(x).
\end{equation*}

Recall that $\rho^{-N_2}\leq |\xi|<\rho^{-N_2-1}$, set $N=N_2-N_1-1>0$ and $\theta=\rho^{-1}>1$. Write $\xi=\xi_0 \theta^{N_2}$, then
$1\leq |\xi_0|<\theta$. Since $|\Phi(t)|\leq 1$, it is easy to see that
\begin{equation*}
  \begin{split}
       & \int \big(\prod_{k=N_1+1}^{\infty}|\Phi(\xi\varphi'(x) \rho^k)|\big) d\mu_{N_1}(x)\\
     = & \int \big(\prod_{k=N_1+1}^{\infty}|\Phi(\xi_0\theta^{N_2}\varphi'(x) \rho^k)\big) d\mu_{N_1}(x) \\
  \leq & \int \big(\prod_{k=1}^N |\Phi(\xi_0 \theta^{k} \varphi'(x))|\big) d\mu_{N_1}(x).
 \end{split}
\end{equation*}

Write $\Xi(\xi)=\int \big(\prod_{k=1}^N |\Phi(\xi_0 \theta^{k} \varphi'(x))|\big) d\mu_{N_1}(x)$, now we arrive at the conclusion that
\begin{equation}\label{I_3-1}
 |\Upsilon_3|\leq M \Xi(\xi).
\end{equation}

\vspace{0.5mm}
\textbf{Step 3.} Estimation of the integral $\Xi(\xi)$. \par
\vspace{0.5mm}
Put $H_1=\min_{t\in \widetilde{K}}(a_l-a_s)\varphi'(t)$ and $H_2=\max_{t\in \widetilde{K}}(a_l-a_s)\varphi'(t)$. Fix a sufficiently small $\epsilon>0$
which will be specified later, denote
\begin{equation*}
 \Lambda(\epsilon)=\big\{x\in \widetilde{K}:\; (a_l-a_s)\varphi^{\prime}(x)\in\Gamma(\epsilon)\big\},
\end{equation*}
where $\Gamma(\epsilon)$ is defined as in Section 2 with $c_0=\xi_0$. Then $\Xi(\xi)$ can be divided into two parts:
\begin{equation*}
 \Xi(\xi)=\Big(\int_{\Lambda(\epsilon)} + \int_{\Lambda(\epsilon)^{\complement}}\Big)
          \big(\prod_{k=1}^N |\Phi(\xi_0\theta^k \varphi'(x))|\big) d\mu_{N_1}(x)
        :=\Xi_1(\xi)+\Xi_2(\xi).
\end{equation*}

For each $1\leq k\leq N$, if $\|\xi_0\theta^k (a_l-a_s)\varphi'(x)\|\geq \frac{1}{2(1+\theta)}$, then
\begin{equation*}
 \begin{split}
   |\Phi(\xi_0\theta^k\varphi'(x))|
 = & \big|\sum_{j=1}^m p_j\exp(2 \pi i a_j \xi_0\theta^k\varphi'(x)\big|\\
 = & |p_1 \exp\big(2 \pi i (a_1-a_s) \xi_0\theta^k \varphi'(x))+ \cdots + p_m \exp\big(2 \pi i (a_m-a_s) \xi_0\theta^k \varphi'(x))|\\
\leq & |p_s+p_l \exp\big(2 \pi i (a_l-a_s) \xi_0\theta^k \varphi'(x))|+1-p_l-p_s \\
\leq & \big|p_s+p_l\exp\big(\frac{\pi i}{1+\theta}\big)\big|+1-p_l-p_s \leq 1-\frac{2 p_l}{1+\theta}:=\delta < 1.
 \end{split}
\end{equation*}

Notice that $N=N_2-N_1-1=O(\log |\xi|)$, thus for each $x\in [0, 1]\setminus\Lambda(\epsilon)$,
\begin{equation*}
 \prod_{k=1}^N |\Phi(\xi_0\theta^k \varphi'(x))|<\delta^{\epsilon N}=|\xi|^{-\frac{\epsilon (1-\beta)\log\delta}{\log \rho}}.
\end{equation*}
As a result, we have $\Xi_2(\xi)=O(|\xi|^{-\frac{\epsilon (1-\beta)\log\delta}{\log \rho}})$.

\vspace{0.5mm}
On the other hand, denote the intervals of length $c_0^{-1} \theta^{-N}$ which cover $\Lambda(\epsilon)$ by $I_1, \ldots, I_Q$. It is known by
Lemma \ref{com-lem} that $Q=O(\exp(\omega(\epsilon)N)$ with $\omega(\epsilon)=-\epsilon\log\epsilon-(1-\epsilon)\log(1-\epsilon)+2\epsilon\log(\theta+2)$. For each $1\leq q\leq Q$, let $J_q=(\varphi')^{-1}((a_l-a_s)^{-1}I_q)$, since $\varphi'$ is strictly increasing, we have
\begin{equation*}
 \Lambda(\epsilon)\subset\bigcup\limits_{q=1}^{Q}J_q.
\end{equation*}
It follows that
\begin{equation*}
 \begin{split}
   \Xi_1(\xi)
    &\leq\sum_{q=1}^Q \int_{J_q} \big(\prod_{k=1}^N |\Phi(\xi_0 \theta^k \varphi'(x))|\big) d\mu_{N_1}(x)\\
    & \ll \exp(\omega(\epsilon)N)\max\limits_{1\leq q\leq Q}\mu_{N_1}(J_q).
 \end{split}
\end{equation*}

We know from the definition of $\mu_{N_1}$ and $J_q$ that, the length of each $J_q$ satisfies $|J_q|=O(\rho^{N})=O(|\xi|^{\beta-1})$,
while the gaps between different atoms of the discrete measure $\mu_{N_1}$ are at least $O(\rho^{N_1})=O(|\xi|^{-\beta})$. Therefore, the number of atoms of $\mu_{N_1}$ contained in each $J_q$ is about $O(|\xi|^{2\beta-1})$, which implies that
\begin{equation*}
 \mu_{N_1}(J_q)=O(|\xi|^{2\beta-1}p_l^{N_1})=O(|\xi|^{2\beta-1}\rho^{\alpha N_1})=O(|\xi|^{(2-\alpha)\beta-1}).
\end{equation*}
So we have
\begin{equation*}
  \Xi_1(\xi)=O(\exp(\omega(\epsilon)N))\cdot O(|\xi|^{(2-\alpha)\beta-1})=O(\exp(|\xi|^{-(\frac{\omega(\epsilon)(1-\beta)}{\log\rho}+1-(2-\alpha)\beta)}).
\end{equation*}
Taking $\epsilon\in(0, \delta)$, since $\frac{\omega(\epsilon)(1-\beta)}{\log\rho}+1-(2-\alpha)\beta>\frac{(1-\beta)\epsilon\log\delta}{\log \rho}$,
it follows that
\begin{equation}\label{I_3-2}
  \Xi(\xi)=\Xi_1(\xi) + \Xi_2(\xi)=O(|\xi|^{-\frac{(1-\beta)\epsilon\log\delta}{\log \rho}}).
\end{equation}
From \eqref{decomposition}-\eqref{I_3-2} and the fact that $\beta\in (1/2, 1)$, we obtain \eqref{osc-int}, where the constant $\gamma$ is given by $\min\{2\beta-1, \frac{(1-\beta)\epsilon\log\delta}{\log \rho}\}$ once we take appropriate $\beta$ and $\epsilon$. The second assertion is by taking
$g\equiv 1$. \qed

\section*{Acknowledgements} The authors would like to thank Aihua Fan, Teturo Kamae, Lingmin Liao and Baowei Wang for many useful discussions and valuable suggestions on an earlier draft of this note.


\end{document}